\documentclass{article}
\usepackage{graphicx} % Required for inserting images

\usepackage{enumitem}
\usepackage{amsmath}
\usepackage{amssymb}
\usepackage{mathtools}

\usepackage[all]{xypic}
\usepackage{multirow}
\usepackage{hhline}
\usepackage{verbatim}
\usepackage{mathdots}
\usepackage{comment}
\usepackage{tikz-cd}
\usepackage[titletoc]{appendix}
\usepackage{graphicx}
\usepackage{mathtools}
\usepackage{extpfeil}
\usepackage{thmtools}
\usepackage{float}
\usepackage{amsmath,amssymb}
\usepackage{amsthm}
\usepackage{leftidx}
\usepackage{mathabx }
\usepackage{amsfonts}
\usepackage{bbm}

\newtheorem{thm}{Theorem}[section]
\newtheorem{lemma}[thm]{Lemma}

\theoremstyle{definition}

\theoremstyle{remark}

\title{Aubert duals of strongly positive representations for metaplectic groups}

\author{
Yeansu Kim\thanks{
Department of Mathematics Education, Chonnam National University, 
Gwangju 61186, Republic of Korea.
Email: ykim@jnu.ac.kr
}
\and
Gyujin Oh\thanks{
Department of Mathematics, Chonnam National University, 
Gwangju 61186, Republic of Korea.
Email: lake2314@naver.com
}
}

\begin{document}

\maketitle
\setcounter{footnote}{0}

%\yeansu{
%\begin{itemize}
%    \item Plan to combine Section 2 and 3.
%    \item If possible, we plan to write the comparison with the linear case.
%    \item Add in the introduction about 'Rank-one reducibility via theta lifts' to point out the difference with the linear case??
%\end{itemize}
%}

\begin{abstract}
We determine the Aubert duals of strongly positive representations of the metaplectic group 
\(\widetilde{Sp}(n)\) over a non-Archimedean local field $F$ of characteristic different from two. Using the classification of Mati\'c and an explicit analysis of Jacquet modules, we describe these duals in terms of precise inducing data. Our results extend known descriptions for classical groups to the metaplectic groups case and clarify the role of Aubert duality for non-linear covering groups, providing a foundation for future applications to the study of unitary representations for those cases. Furthermore, We are able to show that the same method applies to odd general spin groups $GSpin(2n+1)$, yielding an explicit description of Aubert duals in that setting as well.
\end{abstract}

\section{Introduction}
The study of classifying unitary representations of connected reductive groups defined over a $p$-adic field $F$ is one important subject in the Langlands program, which is still widely open. Even constructing unitary non-tempeted representations is not much known. Among the tools, we expect that the Aubert involution provides a tool for constructing unitary representations under a conjecture that the Aubert involution preserves unitarity. Very breifly, the Aubert duals describe the duality structure within the Grothendieck group of admissible representations \cite{A95, A96}. While description of the Aubert duals have been extensively developed for connected reductive groups (See \cite{AM23, M17, M19}), its analogue for non-linear covering groups such as the metaplectic double cover of the symplectic groups remains less understood.

The main goal of this paper is to explicitly determine the Aubert duals of strongly positive representations of \(\widetilde{Sp}(n)\), where \(\widetilde{Sp}(n)\) is the unique non-trivial two-fold central extension of $F$-rational points of a symplectic group $Sp(n):=\mathrm{Sp}(n,F)$ and $F$ is a non-Archimedean local field of characteristic different from two. 
%The representation theory of \(\widetilde{Sp}(n)\) exhibits significant new features compared to that of $Sp(2n)$, primarily due to the genuine nature of representations and the presence of Rao’s cocycle in the group law. 
In \cite{M11}, Matic constructed a complete classification of strongly positive representations of \(\widetilde{Sp}(n)\), showing that they can be realized uniquely as irreducible subrepresentations of certain induced representations of the form (\ref{spdsind}). (See also the appendix of \cite{K15}.)
A natural next step is to determine the Aubert duals of those strongly positive representation. We show that their Aubert duals can be described in terms of explicit data of those induced representations.

Note that we expect possible further strong application on the construction of unitary representations. Very briefly, this project is motivated by analogous results for classical groups such as symplectic groups and special orthogonal groups, where in this case Aubert duality has been shown to preserve unitarity for strongly positive representations and, therefore, it play an essential role in describing unitary representations for those cases \cite{H09, M17}. However, for the metaplectic case, new difficulties arise: the non-linearity of the group complicates the structure of Jacquet modules, and standard techniques must be carefully adapted to the covering setting; we leave this application for future work. Our work should be viewed as the metaplectic counterpart of Matić’s explicit description of Aubert duals for classical groups \cite{M17}. Although we follow \cite{M17} closely, this result does not seem to be explicitly written in the literature. Therefore, we record it here as we believe that it provides a key structural ingredient in the development of the unitary duals for metaplectic groups.

We now describe the content of the paper. In Section \ref{Notation and prelim}, we introduce notations and preliminaries such as properties of Aubert duals. In Section \ref{several lemmas}, we describe several lemmas that are needed to describe the Aubert duals of certain representations. Finally, in Section \ref{main}, we describe explicitly the Aubert duals of strongly positive representations of $\widetilde{Sp}(n)$. We also consider the case of odd general spin groups in Section \ref{GSpin_Aubertduals} and able to show that the structure is similar to the case of metaplectic groups.

\section{Notations and preliminaries}\label{Notation and prelim}

%\begin{itemize}
%    \item DONE!! \yeansu{See \cite[Theorem 1.2]{M11} for the classification of strongly positive representation of $\widetilde{Sp(n)}$.}
%    \item DONE!! \yeansu{See \cite[Section 2]{M11} or \cite[Section 2]{HM10} for the definition of $\widetilde{Sp(n)}$.
%    }
%    \item DONE!!
%    \yeansu{Explain also about notation \(\nu^{a+k}\rho \times \dots \times \nu^{a}\rho. \)}
%    \item DONE! 
%    \yeansu{Write the referecne for the Langlands classification. See \cite{BJ16} }
%    \item DONE!
%    \yeansu{Jacquet modules: write part of chapter 2 in \cite{M13}, especially  Theorem 2.3 and following examples in Section 2.}
%    \item DONE!
%    \yeansu{Write down expected definition and properties of Aubert duals for metaplectic groups. DONE!}
%\end{itemize}

Let \(\mathrm{Sp}(n)\) be the symplectic group of rank $n$ defined over a non-Archimedean local field $F$ of characteristic different from two and let $Sp(n)$ be its $F$-rational points. Let \(\widetilde{Sp}(n)\) be the metaplectic group of rank \(n\), the unique non-trivial two-fold central extension of the symplectic group \(Sp(n)\). In other words, the following holds: 
\[
1 \rightarrow \mu_2 \rightarrow \widetilde{Sp}(n) \rightarrow Sp(n) \rightarrow 1,
\] 
where \(\mu_2 = \{1,-1\}\). The multiplication in \(\widetilde{Sp}(n)\) (which is as a set given by \(Sp(n) \times \mu_2\)) is given by Rao's cocyle. Let \(\widetilde{GL}(n)\) be a double cover of $F$-rational points of a general linear group $GL(n):=\mathrm{GL}(n, F)$, where the multiplication is given by 
\[
(g_1, \epsilon_1)(g_2, \epsilon_2)= (g_1g_2, \epsilon_1\epsilon_2(det\ g_1 , det\ g_2)_F)
\]
with $\epsilon_i \in \mu_2$, where $(\cdot , \cdot)_F$ denotes the Hilbert symbol of the field $F$. Let $\alpha$ denote the character of $\widetilde{GL}(n)$ given by $\alpha(g)=(\operatorname{det} g, \operatorname{det} g)_F=(\operatorname{det} g,-1)_F$.

Let \(\Sigma\) denote the set of roots of \(\widetilde{Sp}(n)\) with respect to a fixed minimal parabolic subgroup and let \(\Delta\) stand for a basis of \(\Sigma\). For \(\Theta \subseteq \Delta\), we let \(\widetilde{P}_\Theta\) be the standard parabolic subgroup of \(\widetilde{Sp}(n)\) corresponding to \(\Theta\), which is defined as the preimage of the corresponding parabolic subgroup \(P\) of \(Sp(n)\). If we write the Levi decomposition \(P=MN\), then the unipotent radical \(N\) lifts to \(\widetilde{Sp}(n)\). Therefore, we have a Levi decomposition \(\widetilde{P}_\Theta = \widetilde{M}_{\Theta}N\). The Levi factor \(\widetilde{M}_\Theta\) is not a product of the form \( \widetilde{GL}(n_1) \times \dots \times \widetilde{GL}(n_k) \times \widetilde{Sp}(n')\), but there is an epimorphism 
\[ 
\widetilde{GL}(n_1) \times \dots \times \widetilde{GL}(n_k) \times \widetilde{Sp}(n') \twoheadrightarrow \widetilde{M}_\Theta.
\] 

For a parabolic subgroup \(\widetilde{P}\) of \(\widetilde{Sp}(n)\) with a Levi factor \(\widetilde{M}\) and a representation \(\sigma\) of \(\widetilde{M}\), we denote by \(i_{\widetilde{M}}(\sigma)\) a normalized parabolically induced representation of \(\widetilde{Sp}(n)\) induced from \(\sigma\). For an admissible finite length representation \(\pi\) of \(\widetilde{Sp}(n)\), the normalized Jacquet module of \(\pi\) with respect to the standard parabolic subgroup having a Levi factor \(\widetilde{M}\) will be denoted by \(r_{\widetilde{M}}(\pi)\).

Sometimes we use the following notation for normalized induced represenations.
Every irreducible geniue representation of \(\widetilde{M}\) is of the form \(\pi_1 \otimes \dots \otimes \pi_k \otimes \sigma\), where the representations \(\pi_1, \ldots , \pi_k\), and \(\sigma\) are all geniue. Representations of \(\widetilde{Sp}(n)\) that are parabolically induced from representations \(\pi_1 \otimes \dots \otimes \pi_k \otimes \sigma \) will be denoted by \(\pi_1 \times \dots \times \pi_k \rtimes \sigma\).

Let ${\rm Irr}(\widetilde{Sp}(n))$ (resp.  ${\rm Irr}(\widetilde{GL}(n))$) be a set of irreducible genuine admissible representations of $\widetilde{Sp}(n)$ (resp. $\widetilde{GL}(n)$). For $\sigma \in {\rm Irr}(\widetilde{Sp}(n))$, let $r_k(\sigma)$ be the normalized Jacquet module of $\sigma$ with respect to the standard maximal parabolic subgroup having a Levi factor $\widetilde{GL}(k) \times \widetilde{Sp}(n-k)$. We define $\mu^{*}(\sigma)$ by
\[
\mu^{*}(\sigma) = \sum_{k=0}^{n} s.s.(r_k(\sigma)),
\]
where $s.s.(r_k(\sigma))$ denotes the semisimplification of $r_k(\sigma)$. Tadic's structure formula for metaplectic groups is fully constructed in \cite[Proposition 4.5]{HM10} and in this paper we use two special related cases from \cite{M13}. (See Lemma \ref{J_sp} and \cite[Theorem 6.1]{M13}.)
%%%

Let \(\rho \in {\rm Irr}(\widetilde{GL}(m))\) be unitary and cuspidal. We say that \([\nu^a\rho, \nu^{a+k}\rho] = \{\nu^a\rho,\nu^{a+1}\rho,\ldots,\nu^{a+k}\}\) is a genuine segment, where \(a\in \mathbb{R}\) and \(k \in \mathbb{Z}_{\geqslant 0}\). We denote by \(\delta([\nu^a\rho, \nu^{a+k}\rho])\) the unique irreducible subrepresentation of \(\nu^{a+k}\rho \times \dots \times \nu^{a}\rho \). Note that \(\delta([\nu^a\rho, \nu^{a+k}\rho])\) is a genuine, essentially square-integrable representation attached to \([\nu^a\rho, \nu^{a+k}\rho]\). For an essentially square-integrable representation \(\delta \in {\rm Irr}(\widetilde{GL}(n))\), there exists a unique \(e(\delta) \in \mathbb{R}\) such that the representation \(\nu^{-e(\delta)}\delta\) is a square-integrable representation. Note that \(e(\delta([\nu^a\rho,\nu^b\rho]))= \frac{a+b}{2}\).

For $\sigma \in {\rm Irr}(\widetilde{GL}(m))$, we let $\widetilde{\sigma}$ be a contragredient representation of $\sigma$.

We recall the subrepresentation version of the Langlands classification. For $1 \leq i \leq k$, suppose that $\delta_i \in {\rm Irr}(\widetilde{GL}(n_i))$ is essentially square-integrable such that $e\left(\delta_1\right) \leq e\left(\delta_2\right) \leq \cdots \leq e\left(\delta_k\right)$. Then the induced representation $\delta_1 \times \delta_2 \times \cdots \times \delta_k$ has a unique irreducible subrepresentation, which we denote by $L\left(\delta_1 \times \delta_2 \times  \dots \times \delta_k\right)$. This irreducible subrepresentation is called the Langlands subrepresentation.
%, and it appears with multiplicity one in the composition series of $\delta_1 \times \delta_2 \times \cdots \times \delta_k$. Every irreducible representation $\pi \in R(G L)$ is isomorphic to some $L\left(\delta_1, \delta_2, \ldots, \delta_k\right)$ and, for a given $\pi$, the representations $\delta_1, \delta_2, \ldots, \delta_k$ are unique up to a permutation among those $\delta_i$ having the same exponents.

Similarly, in \cite[Theorem 3.1]{BJ16}, we write a non-tempered \(\pi \in {\rm Irr}(\widetilde{Sp}(n)) \) as the unique irreducible subrepresentation of the induced representation of the form \(\delta_1 \times \dots \times \delta_k \rtimes \tau\), where \(\tau \in {\rm Irr}(\widetilde{Sp}(n))\) is tempered and \(\delta_i \in {\rm Irr}(\widetilde{GL}(n_i)) (1\leq i \leq k)\) are essentially square-integrable such that \(e(\delta_1) \leq \dots \leq e(\delta_k) < 0\). In this case, we write \(\pi = L(\delta_1 \times \dots \times \delta_k \rtimes \tau)\).

The classification of strongly positive representations for $\widetilde{Sp}(n)$ is fully constructed in \cite{M11} and every genuine strongly positive representation can be realized in a unique way (up to a certain permutation) as the unique irreducible subrepresentation of the induced representation of the following form:
\begin{equation}\label{spdsind}
(\prod_{i=1}^{m} \prod_{j=1}^{k_i} \delta([\nu^{\alpha_i - k_i +j}\rho_i, \nu^{\alpha_j^{(i)}}\rho_i ] )) \rtimes \sigma_{cusp}    
\end{equation} 
where \(\rho_i \in {\rm Irr}(\widetilde{GL}(n_i))\ (1 \leq i \leq m)\) are mutually non-isomorphic, cuspidal, and $\alpha$-self-contragredient(i.e. $\rho_i \simeq \alpha \widetilde{\rho_i}$), \(\sigma_{cusp} \in {\rm Irr}(\widetilde{Sp}(n'))\) is cuspidal, \(\alpha_i > 0\) such that \(\nu^{\alpha_i}\rho_i \rtimes \sigma_{cusp}\) reduces, \(k_i= \lceil \alpha_i \rceil\), where \(\lceil \alpha_i \rceil\) denotes the smallest integer which is not smaller than \(\alpha_i\), and, for \(i=1,\ldots,m\), we have \(-1< \alpha_1^{(i)} < \alpha_2^{(i)} < \dots < \alpha_{k_i}^{(i)}\)  and \(\alpha_j^{(i)} - \alpha_i \in \mathbb{Z}\) for \(j=1,\ldots,k_i\).

We finally recall the following definition and main properties of the Aubert duals from \cite[Theorem 4.2 and 4.3]{BJ16}:

\begin{thm}\label{Aubert_Mp}
Define the operator on the Grothendieck group of admissible representations of finite length of \(\widetilde{
Sp}(n)\) by
\[D_{\widetilde{Sp}(n)} = \sum_{\Theta \subseteq \Delta} (-1)^{|\Theta|} i_{\widetilde{M}_\Theta} \circ r_{\widetilde{M}_\Theta}.\]
Operator \(D_{\widetilde{Sp}(n)}\) has the following properties:
    \begin{enumerate}[label=(\arabic*),ref=(\arabic*)]
    \item \(D_{\widetilde{Sp}(n)}\) is an involution. \label{statement1}
    \item \(D_{\widetilde{Sp}(n)}\) takes irreducible representations to irreducible ones. \label{statement2}
    \item If \(\sigma\) is an irreducible cuspidal representation, then \(D_{\widetilde{Sp}(n)}(\sigma) = (-1)^{|\Delta|} \sigma\). \label{statement3}
    \item For a standard Levi subgroup \(\widetilde{M} = \widetilde{M}_\Theta\), we have\[r_{\widetilde{M}} \circ D_{\widetilde{Sp}(n)} = Ad(w) \circ D_{w^-1(\widetilde{M})} \circ r_{w^-1(\widetilde{M})},\] where w is the longest element of the set $\{w \in W : w^{-1}(\Theta) > 0\}$. \label{statement4}
    \end{enumerate}
\end{thm}

If \(\sigma\) is an irreducible representation of \(\widetilde{Sp}(n)\), we denote by \(\hat{\sigma}\) the representation \(\pm D_{\widetilde{Sp}(n)} (\sigma)\), taking the sign \(+\) or \(-\) such that hat is a positive element in the Grothendieck group of admissible representations of finite length of \(\widetilde{Sp}(n)\). We call \(\hat{\sigma}\) the Aubert dual of \(\sigma\). 

\section{Several lemmas}\label{several lemmas}
In this subsection, we introduce three key lemmas that are used when we prove our main theorem, i.e., Theorem \ref{mainthm_special}.

%\yeansu{Need explain the notation $\hat{\sigma}$ for its Aubert duals}
\begin{lemma}\label{lem1}
Let \(\sigma \in {\rm Irr}(\widetilde{Sp}(n))\) and suppose that 
\(
r_{\widetilde{M}}({\sigma}) \geq \nu^{x_1}\rho_1 \otimes \dots \otimes \nu^{x_m}\rho_m \otimes \sigma_{cusp} 
\), where 
\(\rho_i \in {\rm Irr}(\widetilde{GL}(k))\) and \(\sigma_{cusp} \in {\rm Irr}(\widetilde{Sp}(n'))\) are all cuspidal, and $\widetilde{M}$ is an appropriate standard Levi subgroup. Then 
\begin{equation}\label{eq_lemma}
r_{\widetilde{M}}(\hat{\sigma}) \geq \nu^{-x_1}\alpha\widetilde{\rho}_1 \otimes \dots \otimes \nu^{-x_m}\alpha\widetilde{\rho}_m \otimes \sigma_{cusp}.    
\end{equation}
In particular, if \(\sigma \in{\rm Irr}(\widetilde{Sp}(n))\) is a strongly positive, then in (\ref{eq_lemma}) \(-x_i < 0\) for \(i = 1,\ldots,m\).
\begin{proof}
By assumption \(r_{\widetilde{M}}({\sigma})\) contains 
\[
\nu^{x_1}\rho_1 \otimes \dots \otimes \nu^{x_m}\rho_m \otimes \sigma_{cusp}
\] 
with respect to the appropriate standard Levi subgroup \(\widetilde{M}\). Let $w$ be as in Theorem \ref{Aubert_Mp} such that $w^{-1}(\widetilde{M})=\widetilde{M}.$ Theorem \ref{Aubert_Mp}(4) implies 
\[
r_{\widetilde{M}} \circ D_{\widetilde{Sp}(n)}(\sigma) = Ad(w) \circ D_{\widetilde{M}} \circ r_{\widetilde{M}}(\sigma).
\]
Since \(\rho_1,\ldots,\rho_m,\sigma_{cusp}\) are irreducible geuine cuspidal representations,
\[
D_{\widetilde{M}}(\nu^{x_1}\rho_1 \otimes \dots \otimes \nu^{x_m}\rho_m \otimes \sigma_{cusp}) = \pm \nu^{x_1}\rho_1 \otimes \dots \otimes \nu^{x_m}\rho_m \otimes \sigma_{cusp},
\] 
and
\[
Ad(w)(\pm \nu^{x_1}\rho_1 \otimes \dots \otimes \nu^{x_m}\rho_m \otimes \sigma_{cusp}) = \pm \nu^{-x_1}\alpha\widetilde{\rho}_1 \otimes \dots \otimes \nu^{-x_m}\alpha\widetilde{\rho}_m \otimes \sigma_{cusp},
\]
which completes the first assertion of the lemma. Second assertion follows directly due to definition of strongly positive representations.
\end{proof}
\end{lemma}

\begin{lemma}\label{lem2}
Let \(\sigma \in{\rm Irr}(\widetilde{Sp}(n))\) be strongly positive and let \(\sigma_{cusp}\) be the partial cuspidal support of \(\sigma\). Then \(\hat{\sigma} = L(\delta_1 \times \dots \times \delta_m \rtimes \sigma_{cusp})\), for irreducible genuine essentially square-integrable representations \(\delta_1,\ldots,\delta_m\) of $\widetilde{GL}(n_i)$, such that \(e(\delta_i) \leq e(\delta_{i+1}) < 0\) for \(i=1,\ldots,m-1\).
\end{lemma}
\begin{proof}
By the Langlands classfication \cite[Theorem 3.1]{BJ16}, \(\hat{\sigma} = L(\delta_1 \times \dots \times \delta_m \rtimes \tau)\), for irreducible essentially square-integrable representations \(\delta_1,\ldots,\delta_m\) of $\widetilde{GL}(n_i)$, such that \(e(\delta_i) \leq e(\delta_{i+1}) < 0\) for \(i=1,\ldots,m-1\), and a tempered representation \(\tau \in {\rm Irr}(\widetilde{SP}(n'))\) for some \(n' \leq n\). If \(\tau\) is not isomorphic to \(\sigma_{cusp}\), according to the Casselman criterion \cite[Prop 3.5]{BJ13}, then there is an \(x \geq 0\) and a cuspidal representation \(\rho \in {\rm Irr}(\widetilde{GL}(t))\) such that \(\tau\) is a subrepresentation of \(\nu^x\rho \rtimes \tau'\), for some \(\tau' \in {\rm Irr}(\widetilde{Sp}(n''))\). Using Frobenius reciprocity, together with transitivity of Jacquet modules, we get a contradiction with Lemma \ref{lem1}. This ends the proof.
\end{proof}

We also need the following lemma, which generalizes \cite[Lemma 3.4]{M13} to the case of metaplectic groups:

\begin{lemma}\label{lem_sp}
Let $\sigma \in {\rm Irr}(\widetilde{Sp}(n))$ be strongly positive. Suppose that $\tau \otimes \sigma^{\prime} \leq \mu^{*}(\sigma)$ for some $\tau \in  {\rm Irr}(\widetilde{GL}(t))$ and $\sigma^{\prime} \in  {\rm Irr}(\widetilde{Sp}(n'))$. Then $\sigma^{\prime}$ is strongly positive.
\end{lemma}
\begin{proof}
    The argument is analogous to that in \cite[Lemma 3.4]{M13} and the main ideas are to use the strong positivity of $\sigma$ and the Forbenius reciprocity. We briefly explain its adaptation to metaplectic group case without repeating the whole argument. Suppose that $\sigma'$ is not strongly positive. Then, there exists a supercuspidal representation $\nu^{c_1} \rho_{1} \otimes \cdots \otimes \nu^{c_k} \rho_{k} \otimes \sigma_{\text {cusp }}$ ($\exists$ $j$ such that $c_j \leq 0$) that appears in the Jacquet module of $\sigma$ with respect to the appropriate parabolic subgroup. The Frobenius reciprocity implies that $\sigma$ is a subrepresentation of $\nu^{c_{1}} \rho_{1} \times \cdots \times$ $\nu^{c_k} \rho_{k} \rtimes \sigma_{\text {cusp}}$, which contradicts the strong positivity of $\sigma$.
\end{proof}

\section{Main theorems on the Aubert duals of strongly positive representations for $\widetilde{Sp}(n)$}\label{main}

In this subsection, we determine the Aubert duals of strongly positive representations of $\widetilde{Sp}(n)$. Note that our main idea of the proof follows similarly as in the classical group case and therefore we emphasize how we adapte ideas of the proof in \cite{M17} to metaplectic groups. In case the proof in \cite{M17} is omitted, we provide a detail. 

We first consider the special case, the set of strongly positive representations whose cuspidal supports are the representation $\sigma_{cusp}$ and twists of the representation $\rho$ by positive valued characters, denoted $D(\rho, \sigma_{cusp})$, where $\rho \in {\rm Irr}(\widetilde{GL}(m)) $ is $\alpha$-self-contragredient(i.e. $\rho \simeq \alpha \widetilde{\rho}$) and cuspidal for some $m \in \mathbb{N}$ and $\sigma_{cusp} \in {\rm Irr}(\widetilde{Sp}(n'))$ is cupsidal for some $n' \in \mathbb{N}$. Via the method of the theta correspondence, Hanzer and Muic prove that there is a unique non-negative real number $a$ such that $\nu^s \rho \rtimes \sigma_{cusp}$ ($s \in \mathbb{R})$ reduces if and only if $|s|=a$ \cite{HM11}. We call such $a$ as the reducibility point of $\rho$ and $\sigma_{cusp}$, which we denote by $\alpha_{\rho}$.

\begin{lemma}\label{lem_red_zero}
    When the reducibility point $\alpha_{\rho}$ above is $0$, $\sigma_{cusp}$ is the only strongly positive representation in $D(\rho ; \sigma_{cusp})$. 
\end{lemma}
\begin{proof}
This is exactly as in \cite{M17}. We briefly write how we applied the idea in \cite{M17} to the case of metapletic groups.
Let $\sigma_{sp} \in D(\rho ; \sigma_{cusp})$ be strongly positive and noncuspidal. Then we have 
\[
\sigma_{sp} \hookrightarrow \nu^{x_1} \rho \times \cdots \times \nu^{x_{m-1}} \rho \times \nu^{x_m} \rho \rtimes \sigma_{cusp} \cong \nu^{x_1} \rho \times \cdots \times \nu^{x_{m-1}} \rho \times \nu^{-x_m} \rho \rtimes \sigma_{cusp}
\]
for some $x_i>0$ for $i=1, \ldots, m$ with $m \geq 1$ since $\nu^{x_m} \rho \rtimes \sigma_{cusp} \cong \nu^{-x_m} \rho \rtimes \sigma_{cusp}$ is irreducible. Here, we use our assumption $\rho \simeq \alpha \widetilde{\rho}$. This contradicts that $\sigma_{sp}$ is strongly positive. Therefore, $\sigma_{cusp}$ is the only strongly positive representation contained in $D(\rho ; \sigma_{cusp})$.
\end{proof}

Due to Lemma \ref{lem_red_zero}, we can assume that the reducibility point $\alpha_{\rho}>0$ is positive. Let $k=\lceil \alpha_{\rho} \rceil$. The main results in \cite{M11} imply that there is a bijection between the set of strongly positive representation in $D(\rho ; \sigma_{cusp})$ and the set of all ordered $k$-tuples $\left(a_1, \ldots, a_k\right)$ such that $a_i- \alpha_{\rho} \in \mathbb{Z}$, for $i=1, \ldots, k$, and $-1<a_1<a_2<\ldots<a_k$. Non-cuspidal strongly positive representation corresponding to such $k$-tuple $\left(a_1, \ldots, a_k\right)$ will be denoted by $\sigma_{\left(a_1, \ldots, a_k\right)}$, and it is the unique irreducible subresentation of the following:
\begin{equation}\label{ind1}
\delta([\nu^{\alpha_{\rho}-k+1} \rho, \nu^{a_{1}} \rho]) \times \delta([\nu^{\alpha_{\rho}-k+2} \rho, \nu^{a_{2}} \rho]) \times \cdots \times \delta([\nu^{\alpha_{\rho}} \rho, \nu^{a_k} \rho]) \rtimes \sigma_{cusp}
\end{equation}
We allow $\delta([\nu^{\alpha_{\rho}-k+i} \rho, \nu^{a_{i}} \rho])$ be an empty set in case $\alpha_{\rho}-k+i > a_{i}$. 
%\yeansu{
%OR
%\[
%\delta([\nu^{\alpha-l+1} \rho, \nu^{a_{k-l+1}} \rho]) \times \delta([\nu^{\alpha-l+2} \rho, \nu^{a_{k-l+2}} \rho]) \times \cdots \times \delta(\left[\nu^\alpha \rho, \nu^{a_k} \rho\right]) \rtimes \sigma_{cusp}
%\]
%where $l=k-\min \left\{i: a_i \geq \alpha_{\rho}-k+i\right\}+1$.}

%\yeansu{20250301. It seems that there is a typo \cite[Theorem 6.1]{M13}. The format in \cite{M17} seems to be correct.}

We now recall the following Tadic's structure formula, which is a special case, i.e., $D(\rho ; \sigma_{cusp})$ of \cite[Theorem 6.1]{M13}:

\begin{lemma}[\cite{M13}]\label{J_sp}
Let $\sigma_{(a_1, \ldots, a_k)}$ be an irreducible strongly positive representation in $D(\rho ; \sigma_{cusp})$. Then we have
\[
\mu^*(\sigma_{(a_1, \ldots, a_k)})=\sum L(\delta([\nu^{b_1+1} \rho, \nu^{a_1} \rho]) \times \cdots \times \delta([\nu^{b_k+1} \rho, \nu^{a_k} \rho])) \otimes \sigma_{(b_1, \ldots, b_k)},
\]
where the sum runs over all ordered $k$-tuples $\left(b_1, \ldots, b_k\right)$ such that $b_i - \alpha_{\rho} \in \mathbb{Z}$ and $\alpha_{\rho} - k + i - 1 \leq b_i \leq a_i$ for $i=1, \ldots, k$. Again, we allow $\delta([\nu^{\alpha_{\rho}-k+i} \rho, \nu^{a_{i}} \rho])$ be an empty set in case $\alpha_{\rho}-k+i > a_{i}$.
\end{lemma}

For $\sigma_{\left(a_1, \ldots, a_k\right)} \in D(\rho ; \sigma_{cusp})$, Lemma \ref{lem2} implies that there exist irreducible essentially square-integrable representations $\delta_1, \ldots, \delta_s$ with $e(\delta_i) \leq e(\delta_{i+1})<0$ for $i=1, \ldots, s-1$ such that $\widehat{\sigma_{\left(a_1, \ldots, a_k\right)}}$ is of the form $L(\delta_1 \times \cdots \times \delta_s \rtimes \sigma_{cusp})$. For $i=1, \ldots, s$, we can write $\delta_i=\delta(\left[\nu^{-x_i} \rho, \nu^{-y_i} \rho\right])$ with $x_i>0$ and $y_i>0$ such that $x_i-\alpha_{\rho}, y_i-\alpha_{\rho} \in \mathbb{Z}$, and we denote $x_i-y_i$ by $z_i$.

We now start to describe the explicit data, i.e., its exponents in the Jacquet modules. Let $k':= {\rm min} \{ i : a_i \geq \alpha_{\rho} -k + i \}$, i.e., $k'$ is the minimun index $i$ such that $\delta([\nu^{\alpha_{\rho}-k+i} \rho, \nu^{a_{i}} \rho])$ in (\ref{ind1}) is nonempty.

\begin{lemma}\label{lem_GL}
    \begin{enumerate}
        \item Let $\sigma_{\left(a_1, \ldots, a_k\right)} \in D(\rho ; \sigma_{cusp})$ be as above and let $\sigma_{\left(b_1, \ldots, b_k\right)} \in D(\rho ; \sigma_{cusp})$ be as in Lemma \ref{J_sp} such that 
        \begin{equation}\label{Jacquetmodule_01}
        r_{\widetilde{M}}(\sigma_{\left(a_1, \ldots, a_k\right)}) \geq \nu^{y_1} \rho \otimes \nu^{y_1+1} \rho \otimes \cdots \otimes \nu^{x_1} \rho \otimes \nu^{y_2} \rho \otimes \cdots \otimes \nu^{x_{i-1}} \rho \otimes \sigma_{\left(b_1, \ldots, b_k\right)},
        \end{equation}
        where $\widetilde{M}$ is an appropriate standard Levi subgroup. Then, there exists $j \in\left\{k', \ldots, k-z_i\right\}$ such that $y_i+r=b_{j+r}$ for $r=0,1, \ldots, z_i$. Furthermore, if $j \geq 2$, then $b_j \geq b_{j-1}+2$.
%        \item Lemma 3.2. For $t=1, \ldots, s-1$ we have $x_t>x_{t+1}$.
        \item $x_{m+1}=x_m-1$ and $y_{m+1}<y_m$ for $m=1, \ldots, s-1$. Also, $x_1=a_k$.
%        \item For $i \in\{k', \ldots, k\}$, let $t=\min \left\{j: z_j  +1 \geq i\right\}$. Then $a_i=x_t-i+1$.
    \end{enumerate}
\end{lemma}

\begin{proof}
We follow the arguments in \cite{M17} but write in detail for completeness. We first prove Lemma \ref{lem_GL}(1). Lemma \ref{J_sp} implies $b_l \leq a_l$ ($l=1, \ldots, k$) and
\begin{equation}\label{Jacquetmodule_02}
r_{\widetilde{M'}}(\sigma_{\left(b_1, \ldots, b_k\right)}) \geq \nu^{y_i} \rho \otimes \nu^{y_i+1} \rho \otimes \cdots \otimes \nu^{x_i} \rho \otimes \nu^{y_{i+1}} \rho \otimes \cdots \otimes \nu^{x_s} \rho \otimes \sigma_{\text {cusp }},
\end{equation}
where $\widetilde{M'}$ is an appropriate standard Levi subgroup.

From Lemma \ref{J_sp}, a comparison of the unitary exponents $y_i$ and $b_l$ ($l=1, \ldots, k$) shows that there exists a $j \in\{k', \ldots, k\}$ such that $y_i=b_j$ and $b_j \geq$ $b_{j-1}+2$ if $j \geq 2$. Again, comparing the exponents $y_{i}+1$, $b_j -1$, and $b_{j+1}$, we have $y_i+1=b_{j+1}$ and $b_{j+1}=b_j+1$. Repeating the same arguments, we obtain $y_i+r=b_{j+r}$ for $r=0,1, \ldots, z_i$ and $j \leq k-z_i$. 

We prove Lemma \ref{lem_GL}(2). For $1 \leq m \leq  s-1$, let $\sigma_m \in D\left(\rho ; \sigma_{cusp}\right)$ be such that 
\begin{equation}\label{Jacquetmodule_04}
 r_{\widetilde{M}_m}(\sigma_{\left(a_1, \ldots, a_k\right)}) \geq \nu^{y_1} \rho \otimes \nu^{y_1+1} \rho \otimes \cdots \otimes \nu^{x_1} \rho \otimes \nu^{y_2} \rho \otimes \cdots \otimes \nu^{x_{m-1}} \rho \otimes \sigma_m,
\end{equation}
and 
\begin{equation}
    r_{\widetilde{M}_m'}(\sigma_m) \geq \nu^{y_m} \rho \otimes \nu^{y_m+1} \rho \otimes \cdots \otimes \nu^{x_m} \rho  \otimes \sigma_{m+1},
\end{equation}
where $\widetilde{M}_m$ and $\widetilde{M}_m'$ are appropriate standard Levi subgroups, $\sigma_1 = \sigma_{\left(a_1, \ldots, a_k\right)}$, and $\sigma_s = \sigma_{cusp}$. We use this notation in the remainder of the proof. Note that Lemma \ref{lem_sp} implies that each $\sigma_m$($1 \leq m \leq  s$) is strongly positive. Suppose that there is $t \in\{1, \ldots, s-1\}$ such that $x_t \leq x_{t+1}$. 
%Lemma \ref{lem_sp} implies that there exists a strongly positive representation $\sigma_t \in D\left(\rho ; \sigma_{c u s p}\right)$ such that the Jacquet module of $\sigma_{\left(a_1, \ldots, a_k\right)}$ with respect to the appropriate parabolic subgroup contains
We write $\sigma_t=\sigma_{\left(b_1, \ldots, b_k\right)}$ since $\sigma_t$ is strongly positive. Lemma \ref{lem_GL}(1) implies that there is $j_1 \in\left\{k', \ldots, k-z_t\right\}$ such that $y_t+r=b_{j_1+r}$ for $r=0,1, \ldots, z_t$. (Here, $k'$ is exactly defined as in $\sigma_{\left(a_1, \ldots, a_k\right)}$.)

%Due to Lemma \ref{lem_sp} again, there exists a strongly positive representation $\sigma_{t+1} \in D\left(\rho ; \sigma_{cusp}\right)$ such that the Jacquet module of $\sigma_{\left(b_1, \ldots, b_k\right)}$ with respect to the appropriate parabolic subgroup contains
%$\nu^{y_t} \rho \otimes \nu^{y_t+1} \rho \otimes \cdots \otimes \nu^{x_t} \rho \otimes \sigma_{t+1}.$
Since $b_1<b_2<\ldots<b_k$, Lemma \ref{J_sp} implies that we can also describe $\sigma_{t+1}$ in terms of classification of strongly positive representations as follows:
\[
\sigma_{t+1}=\sigma_{\left(b_1, \ldots, b_{j_1-1}, b_{j_1}-1, b_{j_1+1}-1, \ldots, b_{j_1+z_t}-1, b_{j_1+z_t+1}, \ldots, b_k\right)}.
\]
Since $b_{j_1+z_t+1}-1>b_{j_1+z_t}-1$ and $x_t \leq x_{t+1}$, Lemma \ref{J_sp} implies that $y_{t+1}$ cannot be any of $\{ b_1, b_2, \cdots, b_{j_1-1}, b_{j_1}-1, b_{j_1+z_t}-1 \}$. Therefore, we obtain $y_{t+1} \geq b_{j_1+z_t+1} > b_{j_1 + z_t} = x_t$. This is a contradication since $e\left(\delta_t\right) \leq e\left(\delta_{t+1}\right)$. We conclude that $x_t > x_{t+1}$ for $1 \leq t \leq s-1$.

%Similar to the proof of the previous lemma, for $t=1, \ldots, s$ we denote by $\sigma_t$ a strongly positive discrete series such that the Jacquet module of $\sigma_{\left(a_1, \ldots, a_k\right)}$ with respect to the appropriate parabolic subgroup contains

%$$
%\nu^{y_1} \rho \otimes \nu^{y_1+1} \rho \otimes \cdots \otimes \nu^{x_1} \rho \otimes \nu^{y_2} \rho \otimes \cdots \otimes \nu^{x_{t-1}} \rho \otimes \sigma_t
%$$
If we write $\sigma_m$ as $\sigma_{\left(b_1^{(m)}, \ldots, b_k^{(m)}\right)}$, a comparison of the biggest unitary exponents implies $x_m=b_k^{(m)}$. Especially, we have $x_1=a_k$ and $x_s=\alpha_{\rho}$. For $m=1, \ldots, s$, we define $j_m=1$ if $b_{j-1}^{(m)}=b_j^{(m)}-1$ for all $j=2, \ldots, k$ and $j_m=\max \left\{j: b_{j-1}^{(m)}<b_j^{(m)}-1\right\}$ otherwise. Using Lemma \ref{J_sp} again, we have $y_m=b_{j_m}^{(m)}$ and $\left(b_1^{(m+1)}, \ldots, b_k^{(m+1)}\right)=\left(b_1^{(m)}, \ldots, b_{j_m-1}^{(m)}, b_{j_m}^{(m)}-1, \ldots, b_k^{(m)}-1\right)$ for $m=1, \ldots, s-1$. Therefore, we have $x_{m+1}=x_m-1$ and $j_m \geq j_{m+1}$ for $m=1, \ldots, s-1$. This also implies $b_{j_{m+1}}^{(m+1)} \leq b_{j_m}^{(m)}-1$, i.e., $y_{m+1}<y_m$ for $m=1, \ldots, s-1$. 
%\yeansu{20250915}
%We prove Lemma \ref{lem_GL}(3).

\end{proof}

With Lemma \ref{lem_GL}, we describe the Aubert dual of a strongly positive representation $\sigma_{\left(a_1, \ldots, a_k\right)} \in D(\rho ; \sigma_{cusp})$ as follows: (We write the proof in detail since the proof is omiited in \cite{M17}. This needs to be written since we need to verify that the argument holds for metaplectic groups.)

\begin{thm}\label{mainthm_special}
    The Aubert dual of the strongly positive representation $\sigma_{\left(a_1, \ldots, a_k\right)}$ is the unique irreducible subrepresentation of the induced representation
\[
\left(\prod_{i=1}^k \prod_{j=-a_{k-i+1}}^{-a_{k-i}-2} \delta\left(\left[\nu^{j-i+1} \rho, \nu^j \rho\right]\right)\right) \rtimes \sigma_{cusp}
\]
where $a_0=\alpha_{\rho} - \lceil\alpha_{\rho}\rceil-1$.  
\end{thm}

\begin{proof}
We describe $x_i$ and $y_i$ for $i=1, \cdots, s$.
In the proof of Lemma \ref{lem_GL}, we show that $x_{m+1}=x_m-1$, $x_m = b_k^{(m)}$, $y_m=b_{j_m}^{(m)}$ for $m=1, \ldots, s-1$. Especially $x_1 = a_k$. From the definitions, note that $a_{i+1} = a_i +1$ for $i=1, \cdots, k'-2$. 

If $a_{k-1} < a_k -1$, then 
$j_1=k, y_1 = a_k$, and $(b_1^{(2)}, b_2^{(2)}, \cdots, b_{k-1}^{(2)}, b_k^{(2)}) = (a_1, a_2, \cdots, a_{k-1}, a_k-1)$. If $a_{k-1} < a_k -2$, then $j_2 = k, y_2 = a_{k-1}$, and $(b_1^{(3)}, b_2^{(3)}, \cdots, b_{k-1}^{(3)}, b_k^{(3)}) = (a_1, a_2, \cdots, a_{k-1}, a_k-2)$. If we continue this process $a_k - a_{k-1} -1$ times until we get $(a_1, a_2, \cdots, a_{k-2}, a_{k-1}, a_{k-1}+1)$ and $x_{a_k - a_{k-1} -1}=y_{a_k - a_{k-1} -1}=-a_{k-1}-2$. Up to this step, the process gives the following product:
\[
\prod_{j=-a_{k}}^{-a_{k-1}-2} \nu^{j} \rho.
\]
Next, if $a_{k-2} < a_{k-1} -1$, we continue the same process $a_{k-1} - a_{k-2} -1$ times until we get $(a_1, a_2, \cdots, a_{k-3}, a_{k-2}, a_{k-2}+1, a_{k-2}+2)$ and $x_{a_k - a_{k-2} -2}= -a_{k-2}-3, y_{a_k - a_{k-2} -2}=-a_{k-2}-2$. Up to this step, the process gives the following product:

\[
\prod_{j=-a_{k}}^{-a_{k-1}-2} \nu^{j} \rho \times \prod_{j=-a_{k-1}}^{-a_{k-2}-2} \delta([\nu^{j-1} \rho, \nu^j \rho]). 
\]

If we continue the above process until we get $(a_1, a_1+1, \cdots, a_1+k-2, a_1+k-1)$, this completely provide all information about $x_i$ and $y_i$ for $i=1, \cdots, s$ and completes the proof.
\end{proof}

We now consider the general case. Similarly, let $D\left(\rho_1, \ldots, \rho_m ; \sigma_{\text {cusp }}\right)$ be a set of strongly positive representations whose cuspidal supports are the representation $\sigma_{cusp}$ and twists of the representation $\rho_i(i=1, \ldots, m)$ by positive valued characters where $\rho_i$ is an irreducible $\alpha$-self-contragredient(i.e. $\rho_i \simeq \alpha \widetilde{\rho_i}$) representation of \(\widetilde{GL}(m_i)\) for some $m_i \in \mathbb{N}$ and $\sigma_{cusp}$ is an irreducible cupsidal representation of $\widetilde{Sp}(n')$ for some $n' \in \mathbb{N}$. Let $\sigma$ be a strongly positive representation in $D\left(\rho_1, \ldots, \rho_m ; \sigma_{\text {cusp }}\right)$ and let $\alpha_i$ a unique non-negative real number such that $\nu^{\alpha_i} \rho_i \rtimes \sigma_{\text {cusp }}$ reduces. 
%Minimality of $m$ implies $\alpha_i>0$. 
Let $k_i=\left\lceil\alpha_i\right\rceil$ and $a_0^{(i)}=\alpha_i-\left\lceil\alpha_i\right\rceil-1$, for $i=1, \ldots, m$. Classification of stronlgy positive representation results \cite{K15, M13} implies that for $i=1, \ldots, m$ there exist $a_1^{(i)}, \ldots, a_{k_i}^{(i)}$ such that $-1<a_1^{(i)}<$ $\cdots<a_{k_i}^{(i)}$ and $a_j^{(i)}-\alpha_i \in \mathbb{Z}$ for $j=1, \ldots, k_i$, such that $\sigma$ is the unique irreducible subrepresentation of the induced representation
\begin{equation}\label{sp_generalcase}
\left(\prod_{i=1}^m \prod_{j=1}^{k_i} \delta\left(\left[\nu^{\alpha_i-k_i+j} \rho_i, \nu^{a_j^{(i)}} \rho_i\right]\right)\right) \rtimes \sigma_{cusp}.
\end{equation}

We now explain how we generalize the arguments for the proof of the special case, i.e., $D\left(\rho ; \sigma_{\text {cusp}}\right)$ to any strongly positive representation in $D\left(\rho_1, \ldots, \rho_m; \sigma_{cusp}\right)$. Note that Lemmas \ref{lem1}, \ref{lem2}, and \ref{lem_sp} apply to any strongly positive representations, and Lemma \ref{J_sp} has a general version in \cite[Theorem 6.1]{M13}: 
%\begin{lemma}[\cite{M13}]\label{J_sp_gen}
%    We have the following equality:
    %he following equality holds in $\mathcal{G}^{g e n} \otimes \mathcal{S}$ :
%\[
%\mu_1^*(\sigma)=\sum_{\left(s_1, s_2, \ldots, s_m\right) \in A c c^{\prime}(\sigma)} L\left(\left(s_1, s_2, \ldots, s_m\right)\right) \otimes \sigma_{\left(s_1, s_2, \ldots, s_m\right)}.
%\]
%\end{lemma} 

\begin{thm}\label{mainthm_general}
The Aubert dual of the strongly positive representation $\sigma \in D\left(\rho_1, \ldots, \rho_m; \sigma_{cusp}\right)$ is the unique irreducible subrepresentation of the induced representation
\[
\left(\prod_{i=1}^m \prod_{l=1}^{k_i} \prod_{j=-a_{k_i-l+1}^{(i)}}^{-a_{k_i-l}^{(i)}-2} \delta\left(\left[\nu^{j-l+1} \rho_i, \nu^j \rho_i\right]\right)\right) \rtimes \sigma_{c u s p}
\]
\end{thm}

\begin{proof}
The proof follows the same line as classcial group case, but since the proof for classical group case is omiited in \cite{M17}, we briefly explain main ideas of the proof to show that arguments of the proof also work for metaplectic groups case.
Let $\sigma_{\left(a_1, \ldots, a_k\right)} \in D\left(\rho_1, \ldots, \rho_m ; \sigma_{\text {cusp }}\right)$ be a strongly positive representation as explained in (\ref{sp_generalcase}). Lemma \ref{lem2} implies that there exist essentially square-integrable irreducible representations $\delta_1^{(i)}, \ldots, \delta_{k_i}^{(i)}$ for $i=1, \ldots, m$ such that $\widehat{\sigma_{\left(a_1, \ldots, a_k\right)}}$ is of the form 

\begin{equation}\label{Langlandsquotient_general case}
L\left(\left( \prod_{i=1}^m
 \delta_1^{(i)} \times \cdots \times \delta_{k_i}^{(i)}\right) \rtimes \sigma_{cusp}\right)
\end{equation}
where for $i=1, \ldots, m$ and $j=1, \ldots k_i$ we can write $\delta_j^{(i)}=\delta\left(\left[\nu^{-x_j^{(i)}} \rho_i, \nu^{-y_j^{(i)}} \rho_i\right]\right)$ with $x_j^{(i)}>0$ and $y_j^{(i)}>0$ such that $x_j^{(i)}-\alpha_{i}, y_j^{(i)}-\alpha_{i} \in \mathbb{Z}$.    

First, we apply the whole argument of the proof of Lemma \ref{lem_GL} to $\delta_1^{(1)} \times \cdots \times \delta_{k_1}^{(1)}$ in (\ref{Langlandsquotient_general case}). In other words, we apply the arguments so that only  representations of $GL$ with cuspidal support $\rho_1$ appear in (\ref{Jacquetmodule_01}) and (\ref{Jacquetmodule_04}) and then describe $x_j^{(1)}$ and $y_j^{(1)}$ for $j=1, \ldots, k_1$.
Furthermore, it is well known that for $1 \leq s \neq k \leq m $, $1 \leq p \leq k_s$, and $1 \leq q \leq k_t$, we have $\delta\left(\left[\nu^{-x_p^{(s)}} \rho_s, \nu^{-y_p^{(s)}} \rho_s\right]\right) \times \delta\left(\left[\nu^{-x_q^{(t)}} \rho_t, \nu^{-y_q^{(t)}} \rho_t\right]\right) 
\cong
\delta\left(\left[\nu^{-x_q^{(t)}} \rho_t, \nu^{-y_q^{(t)}} \rho_t\right]\right) \times \delta\left(\left[\nu^{-x_p^{(s)}} \rho_s, \nu^{-y_p^{(s)}} \rho_s\right]\right)$ since $\rho_s \ncong \rho_k$. Then we apply the whole argument of the proof of Lemma \ref{lem_GL} again to representation of index $i=2$ to describe $x_j^{(2)}$ and $y_j^{(2)}$ for $j=1, \ldots, k_2$. We contiue this process for all $1 \leq i \leq m$. This completes the proof.
\end{proof}

%We recall several lemmas from \cite{M13}. The following is from \cite[Lemma 3.4]{M13} and \cite[Section 5]{M13}:
%\begin{lemma}
%    Let $\sigma$ be a strongly positive representation of $\widetilde{Sp(n)}$. Suppose that $\tau$ (resp. $\sigma'$) is an irreducible representation of $\widetilde{GL(t)}$ (resp. $\widetilde{Sp(n-t)}$) such that $\tau \otimes \sigma' \leq r_{\widetilde{M}}(\sigma)$, where $\widetilde{M} \cong \widetilde{GL(t)} \times \widetilde{Sp(n-t)}$. Then $\sigma'$ is strongly positive.
%\end{lemma}

%Mati\'c described the Jacquet modules of strongly positive representations of the metaplectic groups in \cite[Theorem 6.1]{M13}.

%\begin{lemma}[\cite{M13}]
%    We have the following equality:
    %he following equality holds in $\mathcal{G}^{g e n} \otimes \mathcal{S}$ :
%\[
%\mu_1^*(\sigma)=\sum_{\left(s_1, s_2, \ldots, s_m\right) \in A c c^{\prime}(\sigma)} L\left(\left(s_1, s_2, \ldots, s_m\right)\right) \otimes \sigma_{\left(s_1, s_2, \ldots, s_m\right)}
%\]
%\end{lemma} 

\section{The case of odd general spin groups}\label{GSpin_Aubertduals}
We briefly discuss the case of odd general spin groups $GSpin(2n+1)$. Most of the arguments developed in the metaplectic case extend to the case of odd general spin groups with only minor modifications, and we therefore briefly provide the outline.

First, Theorem~\ref{Aubert_Mp} remains valid for $GSpin(2n+1)$, since it is a connected reductive group and the Aubert involution is well-defined in this case.
Lemma~\ref{lem1} also holds for $GSpin(2n+1)$ since the argument of the proof depends on analysis of Jacquet modules and properties of the Aubert duals, although its formulation requires a slight adjustment due to the difference of the Weyl group action on the representations as follows:
\begin{lemma}
Let \(\sigma \in {\rm Irr}(GSpin(2n+1))\) and suppose that 
$
r_{M}({\sigma}) \geq \nu^{x_1}\rho_1 \otimes \dots \otimes \nu^{x_m}\rho_m \otimes \sigma_{c}, 
$ where 
$\rho_i \in {\rm Irr}(GL(k))$ and $\sigma_{c} \in {\rm Irr}((GSpin(2n'+1))$ are all cuspidal, and $M$ is an appropriate standard Levi subgroup. Then 
\begin{equation}
r_{\widetilde{M}}(\hat{\sigma}) \geq \nu^{-x_1} \widetilde{\rho}_1 \omega_{\sigma_{c}} \otimes \dots \otimes \nu^{-x_m}\widetilde{\rho}_m \omega_{\sigma_{c}} \otimes \sigma_{c},    
\end{equation}
where \(\widetilde{\rho_i}\) is a contragredient representation of \(\rho_i\) and $\omega_{\sigma_{c}}$ is the central character of $\sigma_c$. 
\end{lemma}

We now recall the results on the classification of strongly positive representations of odd general spin groups:

\begin{thm}[\cite{K15}]\label{spds_GSpin}
Every strongly positive representation $\sigma \in {\rm Irr}(GSpin(2n+1))$ can be realized in a unique way (up to a certain permutation) as the unique irreducible subrepresentation of the induced representation of the following form:
\begin{equation}\label{spds_GSpin_ind}
(\prod_{i=1}^{m} \prod_{j=1}^{k_i} \delta([\nu^{\alpha_i - k_i +j}\rho_i, \nu^{\alpha_j^{(i)}}\rho_i ] )) \rtimes \sigma_{c}    
\end{equation} 
where \(\rho_i \in {\rm Irr}(GL(n_i))\ (1 \leq i \leq m)\) are mutually non-isomorphic, cuspidal, and essentially self-contragredient(i.e. $\rho_i \simeq \omega_{\sigma_{c}} \widetilde{\rho_i}$), \(\sigma_{c} \in {\rm Irr}(GSpin(2n'+1))\) is cuspidal, \(\alpha_i > 0\) such that \(\nu^{\alpha_i}\rho_i \rtimes \sigma_{c}\) reduces, \(k_i= \lceil \alpha_i \rceil\), where \(\lceil \alpha_i \rceil\) denotes the smallest integer which is not smaller than \(\alpha_i\), and, for \(i=1,\ldots,m\), we have \(-1< \alpha_1^{(i)} < \alpha_2^{(i)} < \dots < \alpha_{k_i}^{(i)}\)  and \(\alpha_j^{(i)} - \alpha_i \in \mathbb{Z}\) for \(j=1,\ldots,k_i\).    
\end{thm}

Lemma \ref{lem2}, Lemma \ref{lem_sp}, Lemma \ref{lem_red_zero}, and Lemma \ref{lem_GL} also hold for $GSpin(2n+1)$ since their proofs depend on Casselman's criterion, Jacquet modules techniuqe, Frobenius reciprocity, and classification of strongly positive representations (Theorem \ref{spds_GSpin}). Especially, Lemma \ref{lem_GL} also depends on combinatorial arguments on the exponents of the representation of $GL$. Theorem \ref{mainthm_special} follows since it depends on Lemma \ref{lem_GL}. Finally, to generalize Theorem \ref{mainthm_special} to Theorem \ref{mainthm_general}, we only need the irreducibility properties of the representations of $GL$. Therefore, Theorem \ref{mainthm_general} follows.
We now briefly mention the odd $GSpin$ version of Lemma \ref{J_sp}. 
\begin{lemma}
Let $\rho \in {\rm Irr}(GL(k))$ be an irreducible cuspidal, and essentially self-contragredient representation and $\sigma_c \in {\rm Irr}(GSpin(2n+1))$ be a cuspidal representation. Let $\sigma:=\sigma_{(a_1, \ldots, a_k)} \in D(\rho ; \sigma_{c})$ be a strongly positive representation similarly defined as in (\ref{ind1}) using Theorem \ref{spds_GSpin}.
%$\sigma_{(a_1, \ldots, a_k)}$ 
Then we have
\[
\mu^*(\sigma_{(a_1, \ldots, a_k)})=\sum L(\delta([\nu^{b_1+1} \rho, \nu^{a_1} \rho]) \times \cdots \times \delta([\nu^{b_k+1} \rho, \nu^{a_k} \rho])) \otimes \sigma_{(b_1, \ldots, b_k)},
\]
where the sum runs over all ordered $k$-tuples $\left(b_1, \ldots, b_k\right)$ such that $b_i - \alpha_{\rho} \in \mathbb{Z}$ and $\alpha_{\rho} - k + i - 1 \leq b_i \leq a_i$ for $i=1, \ldots, k$, and $\alpha_{\rho}$ is a reduciblity point of $\rho$ and $\sigma_c$.
\end{lemma}

\begin{proof}
We briefly explain how we generalize main results in \cite{M13} to the case of odd general spin groups. First, all odd $GSpin$ analogue of the results in \cite[Section 2]{M13} are already constructed in \cite[Section 3]{K15}. Also, \cite[Lemma 3.1, 3.2, and 3.3]{K15} are properties about $GL$ representations and \cite[Lemma 3.4]{K15} is Symplectic group analogue of Lemma \ref{lem_sp}, and its proof depends on strong positivity and the Frobenius reciprocity, which is also true for odd general spin groups. Finally, the main arguments in the proofs of \cite[Proposition 4.1, 4.2, and 4.5]{K15} are using Jacquet moodules techinique and compare its unitary exponents and therefore their odd $GSpin$ analogue exactly follows the ones for odd special orthogonal groups, and we do not repeat the same arguments. This completes the proof.
\end{proof}

%\begin{enumerate}
%    \item Proposition 4.1: 
%    \item Proposition 4.2:
%    \item Definition 4.3:
%    \item Lemma 4.4 : GL group:
%    \item Proposition 4.5:
%\end{enumerate}

Let $\rho_i (i=1, \ldots, m)$ and $\sigma_c$ are as in Theorem \ref{spds_GSpin}. Let $D(\rho_1, \ldots, \rho_m ; \sigma_c)$ be a set of strongly positive representations whose cuspidal supports are $\sigma_c$ and twists of the representation $\rho_i(i=1, \ldots, m)$ by positive valued characters.

%Lemma \ref{J_sp} has a general version in \cite[Theorem 6.1]{M13}: 

\begin{thm}\label{mainthm_GSpin}
Let $\sigma \in D(\rho_1, \ldots, \rho_m ; \sigma_c)$ be a strongly positive representation as in Theorem \ref{spds_GSpin}. Its Aubert dual is the unique irreducible subrepresentation of the following induced representation:
\[
\left(\prod_{i=1}^m \prod_{l=1}^{k_i} \prod_{j=-a_{k_i-l+1}^{(i)}}^{-a_{k_i-l}^{(i)}-2} \delta\left(\left[\nu^{j-l+1} \rho_i, \nu^j \rho_i\right]\right)\right) \rtimes \sigma_c
\]
\end{thm}

%\begin{enumerate}
%    \item Theorem \ref{Aubert_Mp}: exists since $GSpin(2n+1)$ is a connected reductive group.
%    \item Lemma \ref{lem1} holds with slightly diffrent format. 
%    \item Lemma \ref{lem2}, Lemma \ref{lem_sp}, and Lemma \ref{lem_red_zero} also hold for $GSpin(2n+1)$ since the proof depends on Casselman's criterion, Jacquet modules techniuqe, Frobenius reciprocity, and classification of strongly positive representations. 
%    \item Lemma \ref{J_sp}: find the reference for $GSpin(2n+1)$.
%    \item Lemma \ref{lem_GL} also holds for $GSpin(2n+1)$ since $\cdots$
%    \item Theorem \ref{mainthm_special} also holds since it depends on Lemma \ref{lem_GL} and combinatorial arguments on exponents $(a_1, \cdots, a_k)$.
%    \item Theorem \ref{mainthm_general} also holds since it depends on Theorem \ref{mainthm_special} and properties on $GL$ representations.
%\end{enumerate}

\section{Acknowledgement}
The first author would like to thank Ivan Mati\'{c} for answering questions and for helpful discussions. The first author has been supported by the National Research Foundation of Korea (NRF) grant funded by the Korea government (MSIP) (No. RS-2022-0016551 and  No. RS-2024-00415601 (G-BRL)). The second author has been supported by the National Research Foundation of Korea (NRF) grant funded by the Korea government (MSIP) (No. RS-2024-00415601 (G-BRL)).

%\yeansu{Plan to remove Section 4 and 5 soon. It is possible that we will add the content in Section 4 put it in the appendix.}

%\section{Future project: $SO(2n)$ and $GSpin(2n)$}
%Probably plan to co-work with Héctor del Castillo.


\begin{thebibliography}{1000000}

\bibitem{AM23}
H. Atobe and A. Mínguez,
{\it The explicit Zelevinsky-Aubert duality,}
Compos. Math. 
{\bf 159(2)} (2023), 380--418


\bibitem{A95}
A. Aubert, 
{\it Dualité dans le groupe de Grothendieck de la catégorie des représentations lisses de longueur finie d'un groupe réductif p-adique (French, with English summary),} 
Trans. Amer. Math. Soc.
{\bf 347} (1995). no.\,6, 2179--2189, 

\bibitem{A96}
A. Aubert, 
{\it Erratum: "Duality in the Grothendieck group of the category of finitelength smooth representations of a p-adic reductive group", [Trans. Amer. Math. Soc. $\mathbf{3 4 7}$ (1995), no. 6, 2179--2189,}
Trans. Amer. Math. Soc. 
{\bf 348} (1996), no. 11, 4687--4690.
 no.\,4, 381--412.

\bibitem{BJ13}
D. Ban and C. Jantzen, 
{\it The Langlands Quotient Theorem for Finite Central Extensions of $p$-adic Groups I,} 
Glasnik matematički {\bf 48} (2013), 313--334.

\bibitem{BJ16}
D. Ban and C. Jantzen, 
{\it The Langlands Quotient Theorem for Finite Central Extensions of $p$-adic Groups II: Intertwining Operators and Duality,} 
Glasnik matematički {\bf 51}(1) (2016), 153--163.

\bibitem{H09}
M. Hanzer, 
{\it The unitarizability of the Aubert dual of strongly positive square integrable representations,}
Isr. J. Math. {\bf 169} (2009), 251--294.

\bibitem{HM10}
M. Hanzer and G. Mui\'{c},
{\it Parabolic induction and Jacquet functor for metaplectic groups,}
J. Algebra {\bf 323} (2010), 241--260.

\bibitem{HM11}
M. Hanzer and G. Mui\'{c},
{\it Rank one reducibility for metaplectic groups via theta correspondence,}
Canad. J. Math. {\bf 63}(3) (2011), 591--615.

\bibitem{K15}
Y. Kim,
{\it Strongly positive representations of $GSpin_{2n+1}$ and the Jacquet module method. with an appendix, ``Strongly positive representations in an exceptional rank-one reducibility case'' by Ivan Mati\'{c},}
Math. Z. {\bf 279} (2015), 271--296.

%\bibitem{KM19}
%Y. Kim and I. Mati\'c, 
%{\it Classification of strongly positive representation of general unitary groups,} In: Aubert AM., Mishra M., Roche A., Spallone S. (eds) Representations of Reductive $p$-adic Groups, Progr. Math. {\bf 328} (2019), 161--174.  

\bibitem{M11}
I. Mati\'{c},
{\it Strongly positive representations of metaplectic groups,}
J. Algebra {\bf 334} (2011), 255--274.

\bibitem{M13}
I. Mati\'{c},
{\it Jacquet modules of strongly positive representations of the metaplectic groups $\widetilde{Sp(n)}$,}
Tran. Amer. Math. Soc. {\bf 365} (2013), 2755--2778.

\bibitem{M17}
I. Mati\'{c},
{\it Aubert duals of strongly positive discrete series and a class of unitarizable representations,}
Proc. Amer. Math. Soc. {\bf 145} (2017), 3561--3570.

\bibitem{M19}
I. Mati\'{c},
{\it Aubert duals of discrete series: the first inductive step,}
Glas. Mat. Ser. III, {\bf 54}(74) (2019), 133--178.



\end{thebibliography}
\end{document}